\documentclass{amsart}
\usepackage[ngerman, british]{babel}			
\usepackage[utf8]{inputenc}				
\usepackage[T1]{fontenc}				
\usepackage{amsmath,amsthm,amssymb,amsfonts}		
\usepackage{hyperref}					
\usepackage{paralist}

\usepackage{stmaryrd}
\usepackage[
  hmarginratio={1:1},     
  vmarginratio={1:1},     
  textwidth=440pt,        
  heightrounded,          
]{geometry}

\newtheorem{theorem}{Theorem}[section]
\newtheorem{proposition}[theorem]{Proposition}
\newtheorem{corollary}[theorem]{Corollary}
\newtheorem{lemma}[theorem]{Lemma}

\theoremstyle{definition}
\newtheorem{example}[theorem]{Example}
\newtheorem{example*}{Example}
\newtheorem{remark}[theorem]{Remark}

\newcommand{\Z}{\mathbb{Z}}
\newcommand{\Q}{\mathbb{Q}}
\newcommand{\R}{\mathbb{R}}

\newcommand{\N}{\mathbb{N}}

\begin{document}

\title[Partial genera of unimodular lattices over imaginary-quadratic fields]{Steinitz classes and partial genera of unimodular lattices over imaginary-quadratic fields}

\author{Michael Jürgens}
\address{Fakultät für Mathematik, Technische Universität Dortmund, 44221 Dortmund, Germany}
\email{michael.juergens@math.tu-dortmund.de}

\author{Marc C. Zimmermann}
\address{Fakultät für Mathematik, Technische Universität Dortmund, 44221 Dortmund, Germany}
\email{marc-christian.zimmermann@math.tu-dortmund.de}

\date{\today}
\keywords{hermitian unimodular lattice, imaginary-quadratic field, mass formula, Steinitz class}

\begin{abstract}
In this paper we first of all determine all possible genera of (odd and even) definite unimodular lattices over an imaginary-quadratic field. The main questions are whether the partial class numbers of lattices with given Steinitz class within one genus are equal for all occurring Steinitz classes and whether the partial masses of those partial genera are equal. We show that the answer to the first question in general is ``no'' by giving a counter example, while the answer to the second question is ``yes'' by proving a mass formula for partial masses. Finally, we determine a list of all single-class partial genera and show that a partial genus consists of only one class if and only if the whole genus consists of only one class.
\end{abstract}

\maketitle

\section*{Introduction}
Let $K$ be an algebraic number field. If $L$ is a quadratic lattice over $K$, it is easily seen that the Steinitz class $\mathfrak{st}(L)$ is an invariant of the genus of $L$. In particular if one lattice in a genus is free so are all others. In general, this is not the case for hermitian lattices over $K$ as the following three examples of genera of (positive) definite unimodular hermitian lattices over imaginary-quadratic fields show:

\begin{example*} Let $K:=\mathbb{Q}(\sqrt{-47})$. The class group of $K$ is the cyclic group of order 5 with generator $[\mathfrak{a}]$. There is exactly one genus of definite unimodular lattices of rank 4 over $K$: \vspace{0.2cm}
\begin{center}
\renewcommand{\arraystretch}{1.2}
\begin{tabular}{|r|c|ccccc|c|}
 \hline
   &  & $[\mathfrak{o}]$ & $[\mathfrak{a}]$ & $[\mathfrak{a}^2]$ & $[\mathfrak{a}^3]$ & $[\mathfrak{a}^4]$ & parity \\
   &  &  \multicolumn{5}{|c|}{partial} &  \\
 \hline
 class number & 1510 & 302 & 302 & 302 & 302 & 302 & odd \\
 mass   & $\tfrac{1105}{8}$ & $\tfrac{221}{8}$ & $\tfrac{221}{8}$ & $\tfrac{221}{8}$ & $\tfrac{221}{8}$ & $\tfrac{221}{8}$ & \\
 \hline
\end{tabular}
\end{center}
\end{example*}

\begin{example*}Let $K:=\mathbb{Q}(\sqrt{-95})$. The class group of $K$ is the cyclic group of order 8 with generator $[\mathfrak{a}]$. In this case there are exactly two genera of definite unimodular lattices of rank 3 over $K$:\vspace{0.2cm}
\begin{center}
\renewcommand{\arraystretch}{1.2}
\begin{tabular}{|r|c|cccccccc|c|}
 \hline
 &  & $[\mathfrak{o}]$ & $[\mathfrak{a}]$ & $[\mathfrak{a}^2]$ & $[\mathfrak{a}^3]$ & $[\mathfrak{a}^4]$ & $[\mathfrak{a}^5]$ & $[\mathfrak{a}^6]$ & $[\mathfrak{a}^7]$ & parity \\
   &  &  \multicolumn{8}{|c|}{partial} & \\
 \hline
 class number & 428 & 107 & - & 107 & - & 107 & - & 107 & - & odd\\
 mass   & 70 & $\tfrac{35}{2}$ & - & $\tfrac{35}{2}$ & - & $\tfrac{35}{2}$ & - & $\tfrac{35}{2}$ & - & \\
 \hline
 class number & 428 & - & 107 & - & 107 & - & 107 & - & 107 & odd\\
 mass   & 70 & - & $\tfrac{35}{2}$ & - & $\tfrac{35}{2}$ & - & $\tfrac{35}{2}$ & - & $\tfrac{35}{2}$ & \\
 \hline
\end{tabular}
\end{center}
\end{example*}

\begin{example*}Let $K:=\mathbb{Q}(\sqrt{-213})$. The class group of $K$ is the group $\Z/4\Z\times\Z/2\Z$ of order 8 with generators $[\mathfrak{a}]$ and $[\mathfrak{b}]$. There are exactly 4 genera of odd definite unimodular lattices and exactly 2 genera of even definite unimodular lattices of rank $2$ over $K$: \vspace{0.2cm}
\begin{center}
\renewcommand{\arraystretch}{1.2}
\begin{tabular}{|r|c|cccccccc|c|}
 \hline
 &  & $[\mathfrak{o}]$ & $[\mathfrak{a}]$ & $[\mathfrak{b}]$ & $[\mathfrak{ab}]$ & $[\mathfrak{a}^2]$ & $[\mathfrak{a}^3]$ & $[\mathfrak{a}^2\mathfrak{b}]$ & $[\mathfrak{a}^3\mathfrak{b}]$ & parity\\
   &  &  \multicolumn{8}{|c|}{partial} &  \\
 \hline
 class number & 50 & 25 & - & - & - & 25 & - & - & - & odd \\
 mass   & $\tfrac{35}{2}$ & $\tfrac{35}{4}$ & - & - & - & $\tfrac{35}{4}$ & - & - & - &\\
 \hline
 class number & 40 & - & 20 & - & - & - & 20 & - & - & odd\\
 mass   & $18$ & - & $9$ & - & - & - & $9$ & - & - &\\
 \hline
 class number & 74 & - & - & 37 & - & - & - & 37 & - & odd\\
 mass   & $35$ & - & - & $\tfrac{35}{2}$ & - & - & - & $\tfrac{35}{2}$ & - &\\
 \hline
 class number & 84 & - & - & - & 42 & - & - & - & 42 & odd\\
 mass   & $36$ & - & - & - & $18$ & - & - & - & $18$ &\\
 \hline
 \hline
 class number & 32 & - & 16 & - & - & - & 16 & - & - & even\\
 mass   & $12$ & - & $6$ & - & - & - & $6$ & - & - &\\
 \hline
 class number & 56 & - & - & 28 & - & - & - & 28 & - & even\\
 mass   & $\tfrac{70}{3}$ & - & - & $\tfrac{35}{3}$ & - & - & - & $\tfrac{35}{3}$ & - &\\
 \hline
\end{tabular}
\end{center}
\end{example*}
\vspace{0.5cm}
These three examples as well as the explicit enumeration of some other genera of definite unimodular lattices of ``low rank'' over imaginary-quadratic fields with rather ``small discriminant'' lead to the following questions:
\begin{itemize}
 \item[1.] How many genera of definite unimodular lattices of given rank over a given imaginary-quadratic field do exist? In particular, which genera of even definite unimodular lattices do exist?
 \item[2.] How do the Steinitz classes distribute among those genera? Does this distribution induce a partition of the class group?
 \item[3.] Are the \emph{partial class numbers} of lattices within a genus of definite unimodular lattices with the same Steinitz class equal for all occurring Steinitz classes?
 \item[4.] Are the \emph{partial masses} of lattices within a genus of definite unimodular lattices with the same Steinitz class equal for all occurring Steinitz classes?
 \item[5.] What are the single-class partial genera of definite unimodular lattices?

\end{itemize}

In the first section of this paper we collect some basic definitions and facts about hermitian lattices over a number field $K$. We explain the notions of ideal genera as well as special classes and special genera of hermitian lattices. The latter lead to questions similar to 3 and 4 concerning special masses and special class numbers, which have been (partially) answered in the literature (cf. \cite{rehmann},\cite{hashimoto},\cite{hashimoto23}). We describe a connection between special and partial genera, which we use in the following sections to deduce results about partial class numbers and partial masses from these known results.\\
In section 2 we cite a result of Hoffmann (cf. \cite{hoffmann}) and Hashimoto and Koseki (cf. \cite{hashimoto}) concerning the number of hermitian spaces carrying a  definite unimodular lattice. Furthermore, we give a criterion for the existence of an odd resp. even  definite unimodular lattice of given rank and given Steinitz class over a given imaginary-quadratic field. We deduce the essentially known result that two  definite unimodular lattices of same rank and parity are in the same genus if and only if their Steinitz classes are in the same ideal genus. \\
In section 3 we show that the answer to question 3 in general is ``no'' by giving a counter example. Actually, these counter examples seem to be very rare, at least for low rank and small discriminant. Therefore we try to find sufficient conditions on the rank and the class group that imply the equality of the partial class numbers in as many cases as possible. For example we can show that such an equality holds if the rank is 2 or the class group is 2-elementary. Using results of Hashimoto and Koseki (cf. \cite{hashimoto23}) we deduce the same result for lattices of rank 3.\\
In section 4 we use results of Rehmann (cf. \cite{rehmann}) to deduce a mass formula for partial masses which shows that the partial masses are indeed equal for all occurring Steinitz classes within a genus of  definite unimodular lattices, i.e. the answer to question 4 is ``yes''. Furthermore, we evaluate the formulas in Gan and Yu (cf. \cite{gan}) and Cho (cf. \cite{cho}) for the local densities of the genera determined in section 2 and we show how to compute these masses explicitly with the help of generalized Bernoulli numbers.\\
Finally, in section 5 we use the mass formula for the partial masses to compute a list of single-class partial genera. It turns out that a partial genus of definite unimodular lattices consists of exactly one class if and only if the whole genus consists of only one class.\\

All computations described in this paper have been performed with \textsc{Magma} (cf. \cite{magma}). A description of the necessary algorithms, in particular the neighbor method for hermitian lattices can be found in \cite{schiemann}. The implementations of these algorithms in \textsc{Magma} can be found under \url{http://www.mathematik.tu-dortmund.de/sites/michael-juergens/research}

\section{Preliminaries}
Throughout this paper let $K:=\Q(\sqrt{D})$ be an imaginary-quadratic field, i.e. $D<0$ and subsequently we will always assume that $D$ is square-free. We denote by $\overline{\phantom{X}} : K \rightarrow K$,  $\alpha\mapsto \overline{\alpha}$ the non-trivial Galois automorphism of $K/\mathbb{Q}$. Furthermore, let $\mathfrak{o}$ be the ring of integers of $K$ and $d_K$ the discriminant of $K$. It is a well-known fact that $\mathfrak{o}=\Z[\tfrac{1+\sqrt{D}}{2}]$ and $d_K=D$ if $D\equiv_4 1$, and $\mathfrak{o}=\Z[\sqrt{D}]$ and $d_K=4D$ if $D\equiv_4 2,3$. Let $\mathbb{P}$ denote the set of rational prime numbers. A prime $p\in \mathbb{P}$ is called \emph{ramified} (in the extension $K/\Q$) if there is a prime ideal $\mathfrak{p}\subseteq \mathfrak{o}$ with $p\mathfrak{o}=\mathfrak{p}^2$ and \emph{unramified} otherwise. More precisely, an unramified prime $p$ is called \emph{inert} if $p\mathfrak{o}=\mathfrak{p}$ for some prime ideal $\mathfrak{p}$, and \emph{split} if $p\mathfrak{o}=\mathfrak{p}\overline{\
\mathfrak{p}}$ for some prime ideal $\mathfrak{p}$ with $\mathfrak{p}\neq\overline{\mathfrak{p}}$. Let $\chi_K$ denote the character associated to $K$ that is the unique Dirichlet character satisfying
$$\chi_{K}(2) := \begin{cases} 0 & d_K\equiv_4 0\\ 1 & d_K\equiv_8 1\\ -1 & d_K\equiv_8 5 \end{cases} \quad\text{ and }\quad\chi_K(p):=\left(\frac{d_K}{p}\right) \text{ for }p\neq 2. $$
It is also well-known that a prime $p\in \mathbb{P}$ is ramified (resp. inert/split) if and only if $\chi_k(p)=0$ (resp. $-1$/$+1$). We denote by $t$ the number of ramified primes in $K/\mathbb{Q}$, i.e. the number of distinct prime divisors of $d_K$.
We set $I_K$ to be the group of fractional ideals of $K$, $I_K^{(1)}$ the subgroup of fractional ideals of norm $1$ and $\mathcal{C}_K$ to be the ideal class group of $K$. We denote by $h_K$ its class number, and by $[\mathfrak{a}]$ the class of a (fractional) ideal $\mathfrak{a}$. 
\subsection{Lattices and their genera}
By a \emph{hermitian space} $V$ over $K$ we understand a $K$-vector space of dimension $n < \infty$ together with a hermitian form $h: V \times V \rightarrow K$ that is a sesquilinear form (linear in the first component) such that $h(x,y) = \overline{h(y,x)}$. Note that since $\overline{h(x,x)} = h(x,x)$, we have $h(x,x) \in \Q$ for all $x \in V$. A hermitian space $V$ is called \emph{(positive) definite}, if $h(x,x)>0$ for all $x\in V\setminus\{0\}$. We denote by $dV \in \Q^{\ast} / N^K_\mathbb{Q}(K^{\ast})$ the \emph{discriminant} of $V$, by $U(V)$ the \emph{unitary group} of $V$, and we write $V \cong W$ for two \emph{isometric} hermitian spaces. It is well-known that two definite hermitian spaces $V$ and $W$ are isometric if and only if $\operatorname{dim}(V) = \operatorname{dim}(W)$ and $dV = dW$ (cf. \cite{scharlau2012quadratic} Ch. 10, ex. 1.6).\\[1ex]
A \emph{(hermitian) lattice} $L$ over $K$ is a finitely generated $\mathfrak{o}$-module $L \subseteq V$ for some hermitian space $V$. We define the \emph{rank} of $L$ by $\operatorname{rk}(L) := \dim(K{\cdot}L)$ and call $L$ a lattice \emph{on} $V$ if $K{\cdot}L = V$. $L$ is called \emph{definite} if $K{\cdot}L$ is. Every  definite lattice $L$ can be uniquely decomposed into indecomposable lattices $L_1,\ldots,L_s$ (but for their order):
$$ L=L_1\perp \ldots \perp L_s $$
(cf. \cite{eichler} Satz 3). In particular we have the law of cancellation for  definite lattices. On the other hand, every lattice $L$ on $V$ can be written as 
\begin{align}\label{lattice}
L = \mathfrak{a}_1x_1 + \hdots + \mathfrak{a}_nx_n,
\end{align}
with ideals $\mathfrak{a}_1,\ldots,\mathfrak{a}_n$ and a basis $x_1, \hdots, x_n$ of $V$. Let $\mathfrak{st}(L)$ denote the \emph{Steinitz class} of $L$ that is the class $[\mathfrak{a}]\in\mathcal{C}_K$ of the ideal $\mathfrak{a} := \prod_{i=1}^n \mathfrak{a}_i$, which is a well-defined module-invariant of $L$.\\
Two lattices $L$ and $N$ are \emph{isometric}, $L \cong N$, if there exists a $\sigma \in U(V)$ with $\sigma(L) = N$. The set $\operatorname{cls}(L):=\{N\subseteq V\text{ an } \mathfrak{o}\text{-lattice }\mid \exists \sigma\in U(V): N=\sigma(L)\}$ of isometric lattices is called the \emph{(isometry) class} of $L$, and the set of $\sigma \in U(V)$ with $\sigma(L) = L$ is called the \emph{unitary group} of $L$ and denoted $U(L)$. We denote by $L^{\#} := \{x \in V \mid h(x,L) \subseteq \mathfrak{o} \}$ the \emph{dual lattice} of $L$. \\ 
We define the \emph{scale} $\mathfrak{s}L$ and the \emph{norm} $\mathfrak{n}L$ of $L$ to be the $\mathfrak{o}$-ideals generated by $\{h(x,y) \mid x,y \in L\}$ and $\{ h(x,x) \mid x \in L \}$ respectively. For a lattice as in (\ref{lattice}) we define the volume of $L$ by $\mathfrak{v} L := \prod_{i=1}^n \mathfrak{a}_i \overline{\mathfrak{a}_i}  \cdot \operatorname{det}(x_1,\ldots, x_n)$, the latter being the determinant of a Gram matrix of $h$ relative to $x_1,\ldots,x_n$. This is an $\mathfrak{o}$-ideal independent of the choice of the basis and the ideals in (\ref{lattice}) and thus well defined.\\
We call a lattice $L$ on $V$ $\mathfrak{a}$-\emph{modular} for an ideal $\mathfrak{a}$ if $\mathfrak{s}L = \mathfrak{a}$ and $\mathfrak{v} L = \mathfrak{a}^n$. An $\mathfrak{o}$-modular lattice is called \emph{unimodular}. It is easy to show that $L$ is $\mathfrak{a}$-modular if and only if $\mathfrak{a}L^{\#} = L$ and in particular unimodular if and only if $L^{\#} = L$.\\
A lattice $L$ is called \emph{even} if $\mathfrak{n}L \subseteq 2\mathfrak{o}$, and \emph{odd} otherwise, which is referred to as the \emph{parity} of $L$. It is clear that $\mathfrak{n}L = a\mathfrak{o}$ for an $a \in \Z$, and $\operatorname{Tr}^K_\mathbb{Q}(\mathfrak{s}L)\mathfrak{o} \subseteq \mathfrak{n}L \subseteq \mathfrak{s}L$. Since 
\begin{align}\label{trace:o}
 \operatorname{Tr}^K_\mathbb{Q}(\mathfrak{o}) = \begin{cases}
                     \Z & d_K\text{ is odd}, \\
                     2\Z & d_K\text{ is even},
                    \end{cases}
\end{align}
we have $\mathfrak{n}L = \mathfrak{o}$ or $\mathfrak{n}L = 2\mathfrak{o}$ for a unimodular lattice $L$. Thus a unimodular lattice is even if and only if $\mathfrak{n}L = 2\mathfrak{o}$, and odd if and only if $\mathfrak{n}L = \mathfrak{o}$.\\[1ex]
For $p \in \mathbb{P}$ let $\Q_p$ be the localization of $\Q$ at $p$ and $\Z_p$ its ring of integers. Let $K_p := K \otimes_\Q \Q_p$ and $\mathfrak{o}_p := \mathfrak{o} \otimes_\Z \Z_p$ the ring of integers of $K_p$. If $p$ is inert or ramified, then $K_p \cong K_\mathfrak{p} = \Q_p(\sqrt{D})$ is a field and $\mathfrak{o}_p\cong \mathfrak{o}_\mathfrak{p}$, whereas we have $K_p \cong \Q_p \times \Q_p$ and $\mathfrak{o}_p \cong \Z_p \times \Z_p$ if $p$ is split. For $p\in\mathbb{P}$ we make use of the \emph{local norm residue symbol} $(\ \cdot\ ,K/\mathbb{Q})_p : \mathbb{Q}_p^* \rightarrow \{\pm 1\}$ with $(\alpha,K/\mathbb{Q})_p=+1$ if and only if $\alpha\in N^{K_p}_{\mathbb{Q}_p}(K_p^\ast)$. \\
Set $V_p := V \otimes_\Q \Q_p$ equipped with the $p$-adic extension of $h$, and $L_p := L \otimes_\Z \Z_p \subseteq V_p$. The \emph{genus} of a lattice $L$ on $V$ is defined to be
  $$\operatorname{gen}(L) := \{N \subseteq V \text{ an } \mathfrak{o}\text{-lattice }\mid \forall p \in \mathbb{P} : L_p \cong M_p \}.$$
If $\mathcal{G}$ is a genus of unimodular lattices, the \emph{class number} $h_{\mathcal{G}}$ is the number of isometry classes of lattices in $\mathcal{G}$ and if $L_1,\ldots, L_h$ is a set of representatives of the classes contained in $\mathcal{G}$, the \emph{mass} of $\mathcal{G}$ is defined as 
$$ \operatorname{mass}(\mathcal{G}):=\sum_{i=1}^{h_{\mathcal{G}}} \tfrac{1}{|U(L_i)|}. $$
If $[\mathfrak{a}]\in\mathcal{C}_K$ is an occurring Steinitz class of some lattice in $\mathcal{G}$, we denote the \emph{partial genus} of all lattices in $\mathcal{G}$ with Steinitz class $[\mathfrak{a}]$ by $\mathcal{G}_{[\mathfrak{a}]}$. For a partial genus $\mathcal{G}_{[\mathfrak{a}]}$ we denote its \emph{partial class number} by $h_{\mathcal{G}_{[\mathfrak{a}]}}$. As a shorthand notation, we will sometimes use $h_{[\mathfrak{a}]}$ if the associated genus is clear from context.
Assuming that $L_1,\ldots,L_{h_{[\mathfrak{a}]}}$ is a set of representatives for $\mathcal{G}_{[\mathfrak{a}]}$, the \emph{partial mass} of $\mathcal{G}_{[\mathfrak{a}]}$ is defined by
$$ \operatorname{mass}(\mathcal{G}_{[\mathfrak{a}]}):=\sum^{h_{[\mathfrak{a}]}}_{i=1} \tfrac{1}{|U(L_i)|}. $$

\subsection{Relation between special and partial genera}
When replacing the unitary group by the special unitary group in the above definitions one derives the analogous notions of \emph{special class}, \emph{special genus}, \emph{special class number} and \emph{special mass}. They will be denoted by $\operatorname{cls}^+,\operatorname{gen}^+$, $h^+$ and $\operatorname{mass}^+$.\\
It is easy to see that the special classes also put a partition on the set of all lattices on $V$, and that this partition is finer than the partition into isometry classes. Similarly, the partition into special genera is finer than the one into genera. We are interested in connecting partial and special genera. To do so we rely quite heavily on terminology and results from \cite{hashimoto}, for the sake of completeness and smoothing minor differences of notation we reproduce parts of this work to some extent.\\
Given lattices $L,N$ in $V$ we define the ideal $[L:N]$ generated by $\{ \det(\rho) \mid \rho \in \operatorname{End}_K(V):\ \rho(L) \subset N \}$ to be the generalized index of $N$ in $L$ (cf. \cite{shimura}). It then holds that for $L,M,N$ in $V$ we have $[L:M][M:N] = [L:N]$ and for $L$ in $V$ and $\rho \in \operatorname{GL}(V)$ we have $[L:\rho(L) = \det(\rho)\mathfrak{o}$.\\
For lattices $L= \mathfrak{a}_1 x_1 + \ldots + \mathfrak{a}_n x_n$ and $N= \mathfrak{b}_1 y_1 + \ldots + \mathfrak{b}_n y_n$ in $V$ the generalized index is computed to satisfy $[L:N] = \det(\rho)\prod_{i=1}^n \mathfrak{b}_i \mathfrak{a}_i^{-1}$, where $\rho$ is the element of $\operatorname{GL}(V)$ that transforms the basis $x_1,\ldots,x_n$ into $y_1,\ldots,y_n$ (cf. \cite{hoffmann}, after $2.1$). This implies the following characterization.

\begin{lemma} \label{steinitz:index}
 For lattices $L,N$ in $V$ we have that
 \begin{enumerate}
  \item[a)] If both are unimodular: $[L:N] \in I_K^{(1)}$.
  \item[b)] $\mathfrak{st}(L) = \mathfrak{st}(N) \Leftrightarrow [L:N] = (\alpha),\ \alpha \in K^{(1)}$.
 \end{enumerate}
\end{lemma}

\begin{proof}
 Towards $a)$ we note that by the Theorem on invariant factors (cf.  \S $81$ D, \cite{omeara}) there exists a basis $x_1,\ldots,x_n$ of $V$ and ideals $\mathfrak{a}_1,\ldots,\mathfrak{a}_n$, $\mathfrak{e}_1,\ldots,\mathfrak{e}_n$ with $\mathfrak{e}_1 \subseteq \mathfrak{e}_2 \subseteq \ldots \subseteq \mathfrak{e}_n$, such that $L= \mathfrak{a}_1 x_1 + \ldots + \mathfrak{a}_n x_n$ and $N = \mathfrak{a}_1 \mathfrak{e}_1 x_1 + \ldots + \mathfrak{a}_n \mathfrak{e}_n x_n$. Together with the above discussion this implies
 $$ [L:N] = \det(\operatorname{id}_V) \cdot \prod_{i=1}^n \mathfrak{a}_i \mathfrak{e}_i \mathfrak{a}_i^{-1} = \prod_{i=1}^n \mathfrak{e}_i.$$
 Since both lattices are assumed to be unimodular it follows that
 $$ \det(x_1,\ldots,x_n) \cdot \prod_{i=1}^n \mathfrak{a}_i \overline{\mathfrak{a}_i} = \mathfrak{v}L = \mathfrak{o} = \mathfrak{v}N = \det(x_1,\ldots,x_n) \cdot \prod_{i=1}^n \mathfrak{a}_i \mathfrak{e}_i \overline{\mathfrak{a}_i \mathfrak{e}_i }$$
 which implies
 $$ N(\prod_{i=1}^n \mathfrak{a}_i) = N(\prod_{i=1}^n \mathfrak{a}_i) \cdot N(\prod_{i=1}^n \mathfrak{e}_i), $$
 whence 
 $$ N([L:N]) = N(\prod_{i=1}^n \mathfrak{e}_i) = 1. $$ 
 Assumption $b)$ follows from $a)$ and by writing $[L:N] = \det(\rho)\prod_{i=1}^n \mathfrak{b}_i \mathfrak{a}_i^{-1}$ as in the above discussion, from which it follows that $[L:N]$ is pricipal if and only if $\prod_{i=1}^n \mathfrak{b}_i \mathfrak{a}_i^{-1}$ is principal, which is equivalent to  $\mathfrak{st}(L) = \mathfrak{st}(N)$ by $[\prod_{i=1}^n \mathfrak{b}_i \mathfrak{a}_i^{-1}] = [\prod_{i=1}^n \mathfrak{b}_i] \cdot [\prod_{i=1}^n \mathfrak{a}_i]^{-1} = \mathfrak{st}(N) \cdot \mathfrak{st}(L)^{-1}$.
\end{proof}

Let $\mathcal{G} = \operatorname{gen}(L)$ be a genus of unimodular hermitian lattices on $V$. For $M \in \mathcal{G}$ we have noted above that $[L:M] \in I_K^{(1)}$. Given any $\mathfrak{b} \in I_K^{(1)}$ we define, with respect to some fixed $L$, the set $\mathcal{G}_{\mathfrak{b}} = \{ N \in \mathcal{G} \mid [L:N] = \mathfrak{b} \}$ (cf. \cite{hashimoto} $(2.5)$).

We make use of the following result.
\begin{lemma}[cf. \cite{hashimoto}, Lemma 2.4] \label{genus:partition}
  Let $\mathcal{G}$ be a genus of unimodular hermitian lattices. Each $\mathcal{G}_{\mathfrak{b}}$ as above is the disjoint union of exactly $e(\mathcal{G})$ special genera, where $e(\mathcal{G}):=\prod_{p\in\mathbb{P}} [\mathfrak{o}^{(1)}_p : \det{U(L_p)}]$ and $\mathfrak{o}_p^{(1)} :=  \{ \varepsilon \in \mathfrak{o}_p^* \mid \varepsilon \overline{\varepsilon} = 1 \}$, for any $L \in \mathcal{G}$. Furthermore $e(\mathcal{G}) \leq 2$ and $e(\mathcal{G}) = 1$ if the rank or the parity is odd.
\end{lemma}
Since the corresponding Lemma in \cite{hashimoto} is not proven in its entirety we give a proof on the claims about $e(\mathcal{G})$.
\begin{proof}
 The first claim of the Lemma is proven in \cite{hashimoto}. \\
 To prove the claims about $e(\mathcal{G})$ we fix some $L \in \mathcal{G}$. The computation of $e(\mathcal{G})$ simplifies to evaluating $[\mathfrak{o}_2^{(1)}:\det U(L_2)]$, since for any spot $p \in \mathbb{P}$ where $L_p$ splits an one dimensional component the associated index will be $1$, and this happens at all spots $p \ne 2$ in any case for unimodular $L$. Then we can precede any local isometry $\sigma$ between $L,L' \in \mathcal{G}$ by an automorphism of $L$ that can be represented by a matrix of the form $\operatorname{diag}(\varepsilon,1,\ldots,1)$ with respect to the aforementioned splitting and for any $\varepsilon \in \mathfrak{o}_p^{(1)}$. In particular for $\varepsilon = \det(\sigma)^{-1}$ this proves the existence of a local special isometry for $L,L' \in \mathcal{G}$. \\ 
 In particular if the rank or parity of $L$ is odd this implies directly $e(\mathcal{G})=1$, since then $L_2$ does posses an orthogonal basis.\\ 
 For $L$ even $[\mathfrak{o}_2^{(1)}:\det U(L_2)] \le 2$ follows from Proposition $4.18$ and Lemma $4.16$ in \cite{shimura} given the $2\mathfrak{o}_2$-maximality of $L_2$ which is implied by its unimodularity. 
\end{proof}

The above sets $\mathcal{G}_{b}$ are deeply connected to partial genera, we will describe this connection in the subsequent discussion.
\begin{lemma} \label{partial:genus:description}
 Let $\mathcal{G} = \operatorname{gen}(L)$ be a genus of unimodular hermitian lattices on $V$ with $\mathfrak{st}(L) = [\mathfrak{a}]$. Then
 \begin{enumerate}
  \item[a)] $\mathcal{G} = \bigcup_{\mathfrak{b} \in I_K^{(1)}} \mathcal{G}_{\mathfrak{b}}$,
  \item[b)] $\mathcal{G}_{[\mathfrak{a}]} = \bigcup_{\mathfrak{b} \in [\mathfrak{o}]\cap I_K^{(1)}} \mathcal{G}_{\mathfrak{b}}$.
 \end{enumerate}
\end{lemma}

\begin{proof}
 Part $a)$ is Lemma $2.1$ in \cite{hashimoto}, part $b)$ is then a consequence of Lemma \ref{steinitz:index}.
\end{proof}

Now let $U(V)^u := \{ \sigma \in U(V) \mid \det(\sigma) \in \mathfrak{o}^*\}$ (cf. \cite{hashimoto}, $(2.6)$). The following will show that a $U(V)^u$ class of a lattice $L$ is the union of all special classes associated to $L$ inside one of the above $\mathcal{G}_{\mathfrak{b}}$.

\begin{lemma} \label{partial:special}
 Let $\mathcal{G} = \operatorname{gen}(L)$ be a genus of unimodular hermitian lattices on $V$ with $\mathfrak{st}(L) = [\mathfrak{a}]$ and fix an arbitrary $\mathfrak{b} \in [\mathfrak{o}] \cap I_K^{(1)}$ with $\mathcal{G}_{\mathfrak{b}}$ definied with respect to $L$. Then there is a set of lattices $L_1,\ldots,L_k$ that form a set of representatives of the classes in $\mathcal{G}_{[a]}$ such that each $L_i \in \mathcal{G}_{\mathfrak{b}}$. To be precise we find a bijection $\mathcal{G}_{[\mathfrak{a}]} / U(V) \cong \mathcal{G}_{\mathfrak{b}} / U(V)^u$ (independent of the choice of $\mathfrak{b}$), that is the above lattices are also representatives of the $U(V)^u$ classes in $\mathcal{G}_{\mathfrak{b}}$.
\end{lemma}

\begin{proof}
 Let $\mathcal{G}_{\mathfrak{b}}$ be definied with respect to $L$. Choose any representative set $L_1,\ldots,L_k$ of the classes of lattices in $\mathcal{G}_{[a]}$. Set $\mathfrak{d}_i := [L:L_i]$. If not $\mathfrak{d}_i = \mathfrak{b}$ we choose $\rho \in U(V)$ such that $\mathfrak{b} = [L:\rho(L_i)] = \det(\rho)\mathfrak{d}_i$, and replace $L_i$ by $\rho(L_i)$. The latter is possible since $\det(U(V)) = K^{(1)}$ and in any case $\mathfrak{b}$ and $\mathfrak{d}_i$ are principal fractional ideals of norm $1$ (cf. Lemma \ref{steinitz:index}).\\ 
 Let $\operatorname{cls}(M)=\operatorname{cls}(M')$ for $M,M' \in \mathcal{G}_{\mathfrak{b}}$. In this case $[M:M'] = \mathfrak{o}$, since $[L:M'] = [L:M][M:M']$ and $[L:M]=\mathfrak{b} = [L:M']$. Now there is a $\sigma \in U(V)$ such that $\sigma(M) = M'$, and by the above $\mathfrak{o} = [M:M'] = [M:\sigma(M)] = \det(\sigma)\mathfrak{o}$, implying $\det(\sigma) \in \mathfrak{o}^*$, but this just says that $\sigma \in U(V)^u$. Thus the above set of representatives of classes of $L$ in $\mathcal{G}_{[\mathfrak{a}]}$ is also a set of representatives of $U(V)^u$ classes of $L$ in $\mathcal{G}_{\mathfrak{b}}$.
\end{proof}

\begin{lemma} \label{class:partition}
Let $L$ be a unimodular lattice on $V$. Then the $U(V)^u$ class of $L$ is the disjoint union of exactly $e(L):=[\mathfrak{o}^*:\det{U(L})]$ special classes.
\end{lemma}

\begin{proof}
 This is contained in the proof of Proposition $2.5$ in \cite{hashimoto}.
\end{proof}

\begin{corollary} \label{class:number:coincide}
 If $d_K \notin\{ -3, -4\}$, the special and partial class number associated to a unimodular lattice of odd rank coincide.
\end{corollary}
\begin{proof}
Let $L$ be a lattice of odd rank $n$ with $\mathfrak{st}(L) = [\mathfrak{a}]$. We note that $-\operatorname{id} \in U(L)$ and $\det(-\operatorname{id}) = (-1)^n=-1$. Since $d_K \notin\{ -3,-4\}$, we have $|\mathfrak{o}^*| = 2$, and therefore $e(L)=[\mathfrak{o}^*:\det{U(L})] = 1$. Moreover Lemma \ref{genus:partition} states that for each $\mathfrak{b} \in [\mathfrak{a}] \cap I_K^{(1)}$ the set $\mathcal{G}_{\mathfrak{b}}$ is in fact a single special genus. From Lemma \ref{partial:special} we then get that $h_{[\mathfrak{a}]} = |\mathcal{G}_{\mathfrak{b}} / U(V)^u| = |\mathcal{G}_{\mathfrak{b}} / SU(V)| = h^+(\operatorname{gen}^+(L))$, since $U(V)^u$ and special class of lattices in $\mathcal{G}$ coincide following Lemma \ref{class:partition}.
 \end{proof}

\subsection{Ideal genera and genus characters}
A \emph{genus character} of $K$ is a real-valued character of the class group that is a homomorphism $\chi: \mathcal{C}_K \rightarrow \{ \pm 1 \}$.
\begin{proposition}[cf. \cite{zag} §12 Satz 1]  \label{genus:equiv}
 Let $[\mathfrak{a}], [\mathfrak{b}] \in \mathcal{C}_K$, then the following three statements are equivalent:
 \begin{enumerate}
  \item[a)] $\chi([\mathfrak{a}]) = \chi([\mathfrak{b}])$ for every genus character $\chi$ of $K$.
  \item[b)] There is a $\lambda \in K$ with $N(\mathfrak{a}) = N(\lambda) \cdot N(\mathfrak{b})$.
  \item[c)] There is a $[\mathfrak{c}] \in \mathcal{C}_K$ with $[\mathfrak{a}] = [\mathfrak{c}]^2 \cdot [\mathfrak{b}]$. \qed
 \end{enumerate}
\end{proposition}

If the equivalent conditions of \ref{genus:equiv} are fulfilled we say that $[\mathfrak{a}], [\mathfrak{b}] \in \mathcal{C}_K$ belong to the same \emph{ideal genus}. The set of all ideal classes in the same genus as $[\mathfrak{a}]$ is denoted by $\llbracket[\mathfrak{a}]\rrbracket$ or simply by $\llbracket\mathfrak{a}\rrbracket$.\\ 

Since $|\mathcal{C}_K/\mathcal{C}^2_K|=2^{t-1}$ (cf. \cite{zag} §12 Korollar), there are exactly $2^{t-1}$ ideal genera and since ideal genera are precisely the cosets of $\mathcal{C}_K^2$ in $\mathcal{C}_K$, all ideal genera contain the same number of classes, namely $\tfrac{h_K}{2^{t-1}}$. Furthermore, the set of genus characters is precisely the dual group of the (abelian) group $\mathcal{C}_K/\mathcal{C}^2_K$ and therefore there are exactly $2^{t-1}$ genus characters.

\begin{lemma}
 Let $[\mathfrak{a}] \in \mathcal{C}_K$. Then we have $(N(\mathfrak{a}),K/\mathbb{Q})_p = 1$  for all unramified primes $p\in \mathbb{P}$ and moreover
 \begin{align*}
  \prod_{p \text{ ramified}} (N(\mathfrak{a}),K/\mathbb{Q})_p=1.
 \end{align*}
\end{lemma}

\begin{proof}
 It is well known that in the case of a quadratic field $(\alpha,K/\mathbb{Q})_p = (\alpha,D)_{p}$, the latter being the Hilbert symbol at $p$. Let $p$ be an unramified prime, thus $D$ is a unit in $\Q_p$.
 Because there is always an integral ideal in $[\mathfrak{a}]$ which is coprime to $p\mathfrak{o}$ we can assume that $\gcd( N(\mathfrak{a}),p) = 1$. For $p \neq 2$ we are done, because the Hilbert symbol of two units always equals $1$ in that case. If $p = 2$ we have $D \equiv_4 1$, which now implies the same as before. \\ 
 The product formula now follows immediately from the Hilbert reciprocity law.
\end{proof}

Now consider the character $\chi_p : \mathcal{C}_K \rightarrow \{\pm1\}$, $\left[ \mathfrak{a} \right] \mapsto (N(\mathfrak{a}),K/\mathbb{Q})_p$, which is a genus character for every $p \in \mathbb{P}$. By the above Lemma we have $\chi_p \equiv 1$ for every unramified $p$. \\
The image of the homomorphism \vspace{0.2cm}
\begin{center} $\mathcal{X} : \mathcal{C}_K \rightarrow \{\pm1 \}^{t}$, $\left[ \mathfrak{a} \right] \mapsto \left( \chi_p\left(\left[ \mathfrak{a} \right] \right) \right)_{p \text{ ramified}}$,
\end{center} \vspace{0.2cm}
is a subset of $\mathcal{H}:=\{ (x_1,\ldots,x_t)\in\{\pm1 \}^{t} \mid \prod x_i = 1 \}$. To determine the kernel of $\mathcal{X}$, let $[\mathfrak{a}] \in \operatorname{ker}\mathcal{X}$ be arbitrarily chosen, then $(N(\mathfrak{a}),K/\mathbb{Q})_p = 1$ for all $p$, so that $N(\mathfrak{a})$ is a norm for each $p \in \mathbb{P}$. By the Hasse norm Theorem (cf. \cite{omeara} 65:23), there exists $\lambda \in K$ such that $N(\mathfrak{a}) = N(\lambda) = N(\lambda) \cdot N(\mathfrak{o})$. By Proposition \ref{genus:equiv} this is equivalent to $[\mathfrak{a}] = [\mathfrak{c}]^2$ for some $[\mathfrak{c}] \in \mathcal{C}_K$, thus $[\mathfrak{a}] \in \mathcal{C}_K^2$ and $\operatorname{ker}\mathcal{X} \subseteq \mathcal{C}_K^2$. 
The inclusion $\mathcal{C}_K^2 \subseteq\operatorname{ker}\mathcal{X}$ is obvious, therefore $\mathcal{X}$ induces an injective homomorphism: 
\begin{align*}
\overline{\mathcal{X}} : \mathcal{C}_K / \mathcal{C}_K^2 \hookrightarrow \mathcal{H}.
\end{align*}\vspace{0.2cm}
Since there are $2^{t-1}$ ideal genera, i.e. $|\mathcal{C}_K/\mathcal{C}_K^2|=2^{t-1}$, the map $\overline{\mathcal{X}}$ is bijective. In particular, the map $\mathcal{X} : \mathcal{C}_K \rightarrow \mathcal{H}$ is surjective and we have proved:
 \begin{proposition} \label{cor:ideal:-1at2}
  If $d_K$ is even, $d_K \notin\{ -4,-8\}$, there is an ideal class $[\mathfrak{a}]\in \mathcal{C}_K$, such that $$(N(\mathfrak{a}),K/\mathbb{Q})_2 = -1.$$
 \end{proposition}
\begin{proof}
 The condition that $d_K$ is even and $d_K\notin \{-4,-8\}$ implies that $2$ is ramified and $t\geq 2$. Since $\mathcal{X} : \mathcal{C}_K \rightarrow \mathcal{H}$ is surjective, there is a $[\mathfrak{a}]\in\mathcal{C}_K$ with $\mathcal{X}([\mathfrak{a}])=(-1,-1,1,\ldots,1)$, i.e. $\chi_2([\mathfrak{a}])=-1$.
\end{proof}

\section{Genera of unimodular lattices} \label{genera:lattices}
For the sake of self-containment we start with the characterization of those hermitian spaces that carry unimodular lattices, which is a well-known result. We also derive some necessary (and indeed sufficient as Theorem \ref{start:lattices} will show) conditions for a hermitian space to carry an even unimodular lattice.

\begin{proposition}\label{hermitian_spaces} Let $K = \Q(\sqrt{D})$ and $n \in \N$.
 \begin{itemize}
   \item[a)] There are exactly $2^{t-1}$  definite hermitian spaces of dimension $n$ over $K$ (up to isometry) containing a unimodular lattice of rank $n$. There is a bijection between those hermitian spaces and the ideal genera in $\mathcal{C}_K$ given by $V\mapsto \llbracket \mathfrak{st}(L)\rrbracket$ where $L$ is any unimodular lattice $L$ on $V$.  
   \item[b)] For each definite hermitian space $V$ there are at most two genera of unimodular lattices on $V$. For each parity there is at most one such genus.
   \item[c)] For the existence of an even definite unimodular lattice $L$ of rank $n$ on $V$, with $\mathfrak{st}(L) = [\mathfrak{a}]$, the following restrictions apply:
 \begin{enumerate}[i)]
 \item[i)] $d_K$ must be even.
 \item[ii)] $n$ must be even.
 \item[iii)] If $D \equiv_4 3$ then $(N(\mathfrak{a}),K/\mathbb{Q})_2=(-1)^{\frac{n}{2}}$.
 \end{enumerate}
 \end{itemize}
\end{proposition}
\begin{proof}
a) See \cite{hoffmann} 3.7 or \cite{hashimoto} 6.4. b) See \cite{hashimoto} 6.4. c) Formula (\ref{trace:o}) shows that $d_K$ is necessarily even and therefore $p=2$ is ramified. By \cite{jakobowitz} Proposition $10.3$ we have $L_2 \cong \mathbb{H} \perp \hdots \perp \mathbb{H} \perp N$, where $N$ is a unimodular $\mathfrak{o}_2$-lattice of rank $1$ or $2$ with $\mathfrak{n}L_2 = \mathfrak{n}N$. Since there are no even unimodular $\mathfrak{o}_2$-lattices of rank $1$, we conclude $\operatorname{rk}(N) = 2$. This implies that $n$ is even.
If $D \equiv_4 3$ we are in the ramified-unit (RU) case (cf. \cite{jakobowitz} Chapter 5), for which \cite{jakobowitz} Proposition $9.2$ shows that $N \cong \mathbb{H}$ that is $(-1)^{\frac{n}{2}} \in dV_2$. On the other hand $N(\mathfrak{a}) \in dV$, since $\mathfrak{o} = \mathfrak{v}L = \mathfrak{a}\overline{\mathfrak{a}}\cdot \det(x_1,x_2)$, and $\det(x_1,x_2), N(\mathfrak{a})$ are positive rational numbers, we conclude $N(\mathfrak{a}) \cdot \det(x_1,x_2) = 1$. Therefore
\begin{align*}
(N(\mathfrak{a}),K/\mathbb{Q})_2 =  ((-1)^{\frac{n}{2}},K/\mathbb{Q})_2 = ((-1)^{\frac{n}{2}},D)_2= (-1)^{\frac{n}{2}},
\end{align*}
where $D \equiv_4 3$ implies that $(-1,D)_2 = (-1,3)_2 = -1$ or $(-1,D)_2 = (-1,7)_2 = -1$.
\end{proof}

The $2^{t-1}$ definite hermitian spaces containing a unimodular lattice are given by $\langle 1, \hdots, 1, \frac{1}{N(\mathfrak{a})} \rangle$, where $[\mathfrak{a}]$ runs through a set of representatives of the $2^{t-1}$ ideal genera.\\
As an easy consequence we get:
\begin{corollary}
Let $L,N$ be two  definite unimodular lattices of the same rank and parity. Then $L$ and $N$ are in the same genus if and only if $\mathfrak{st}(L)$ and $\mathfrak{st}(N)$ are in the same ideal genus, i.e.
$$ \operatorname{gen}(L)=\operatorname{gen}(N) \quad\Leftrightarrow\quad \llbracket\mathfrak{st}(L)\rrbracket=\llbracket \mathfrak{st}(N)\rrbracket. $$ 
\end{corollary}
\begin{proof}
If $L$ and $N$ are lattices in the same genus, they are lattices on the same hermitian space. By Prop. \ref{hermitian_spaces}a) the Steinitz classes of $L$ and $N$ are in the same genus. On the other hand if $L$ and $N$ are unimodular lattices of the same rank with $\llbracket \mathfrak{st}(L)\rrbracket=\llbracket\mathfrak{st}(N)\rrbracket$, by Prop. \ref{hermitian_spaces}a) $L$ and $N$ are lattices on the same hermitian space. Since the parity of $L$ and $N$ coincides, $L$ and $N$ are in the same genus by Prop. \ref{hermitian_spaces}b).  
\end{proof}

Thus a genus of  definite unimodular lattices over an imaginary-quadratic field is entirely determined by the parity, rank and ideal genus of the Steinitz class of any lattice in it. We denote the genus corresponding to the parity odd resp. even, the rank $n\in\N$ and the ideal genus $\llbracket \mathfrak{a} \rrbracket$ by $\mathbb{I}_{n,\llbracket\mathfrak{a}\rrbracket}$ resp. $\mathbb{II}_{n,\llbracket\mathfrak{a}\rrbracket}$.\vspace{0.2cm}

We will now show that the necessary conditions of Proposition \ref{hermitian_spaces} are also sufficient, i.e. we will give a construction for a sample lattice $L^{\text{odd}}_{n,[\mathfrak{a}]}$ resp. $L^{\text{even}}_{n,[\mathfrak{a}]}$ in each possible genus of (odd or even)  definite unimodular lattices.

\begin{theorem} \label{start:lattices}
Let $[\mathfrak{a}] \in \mathcal{C}_K$ and $n \in \N$.
\begin{enumerate}
 \item[a)] There is a  definite odd unimodular lattice of rank $n$ with Steinitz class $[\mathfrak{a}]$ over $K$.
 \item[b)] There is a  definite  even unimodular lattice of rank $n$ with Steinitz class $[\mathfrak{a}]$ over $K$ if and only if the following conditions are satisfied:
\begin{itemize}
 \item[i)] $d_K$ is even.
 \item[ii)] $n$ is even.
 \item[iii)] If $D \equiv_4 3$, then $(N(\mathfrak{a}),K/\mathbb{Q})_2=(-1)^{\frac{n}{2}}$   
\end{itemize}
\end{enumerate}
\end{theorem}

\begin{proof}
a) The lattice $L^{\text{odd}}_{1,[\mathfrak{a}]}:=\mathfrak{a}x_1$ on the hermitian space $\langle \tfrac{1}{N(\mathfrak{a})} \rangle$ (in the basis $x_1$) is clearly an odd  definite unimodular lattice of rank $1$ with Steinitz class $[\mathfrak{a}]$. Therefore the lattice 
$$ L^{\text{odd}}_{n,[\mathfrak{a}]} := L^{\text{odd}}_{1,[\mathfrak{o}]} \perp\ldots \perp  L^{\text{odd}}_{1,[\mathfrak{o}]}\perp L^{\text{odd}}_{1,[\mathfrak{a}]}$$
is an odd  definite unimodular lattice of rank $n$ and Steinitz class $[\mathfrak{a}]$.\\[1ex]
b) The necessity of the conditions is an immediate consequence of Proposition \ref{hermitian_spaces}$c)$. Let $[\mathfrak{a}]$ be an ideal class as allowed by the above restrictions. We start with $n = 2$:\\
W.l.o.g. assume that $\mathfrak{a}$ is an integral ideal such that $2 \nmid N(\mathfrak{a})$. We claim that there exist $\lambda\in\N$ and $\alpha \in \mathfrak{o}$, such that $L := \mathfrak{o} x_1 + \mathfrak{a} x_2$ on the hermitian space given by the Gram matrix $\left( \begin{smallmatrix} 2 & \alpha \\ \overline{\alpha} & \frac{2 \lambda}{N(\mathfrak{a})}  \end{smallmatrix} \right)$  (in the basis $x_1,x_2$) is a binary even unimodular lattice with Steinitz class $[\mathfrak{a}]$: For $N(\mathfrak{a}) \equiv_4 -1$ we choose $\alpha := 1$ and $\lambda := \frac{1+N(\mathfrak{a})}{4}$ and for $N(\mathfrak{a}) \equiv_4 1$ we choose $\alpha := 1 + \sqrt{D}$ and $\lambda := \frac{1 + (1+D)N(\mathfrak{a})}{4}$. In the latter case conditions $i)$ and $iii)$ imply $D \equiv_4 2$, so that in either case $\lambda \in \N$. Now consider the volume and scale of $L$:
\begin{align*}
\mathfrak{v}L = \mathfrak{a}\overline{\mathfrak{a}}\cdot \det(x_1,x_2) = \mathfrak{a}\overline{\mathfrak{a}}\cdot \tfrac{1}{N(\mathfrak{a})} = \mathfrak{o} ,\quad
\mathfrak{s}L = 2\mathfrak{o} + \alpha\overline{\mathfrak{a}} + \overline{\alpha}\mathfrak{a} +  \tfrac{2\lambda}{N(\mathfrak{a})}\mathfrak{a}\overline{\mathfrak{a}} \subseteq \mathfrak{o}.
\end{align*}
Since $\mathfrak{o} = \mathfrak{v}L \subseteq (\mathfrak{s}L)^2 \subseteq \mathfrak{o}$, we conclude $\mathfrak{s}L = \mathfrak{o}$. Therefore $L$ is a unimodular lattice. Since
\begin{align*}
\mathfrak{n}L = 2\mathfrak{o} + 2\lambda\mathfrak{o} + \operatorname{Tr}^K_\Q(\mathfrak{o}) = 2\mathfrak{o},
\end{align*}
$L$ is even. Clearly the Steinitz class of $L$ is $[\mathfrak{a}]$ and $L$ is  definite. Set $L^{\text{even}}_{2,[\mathfrak{a}]}:=L$. \\[1ex]
For arbitrary even $n$ we split the proof depending on $D$ mod $4$.
For $D \equiv_4 2$ we construct $L^{\text{even}}_{2,[\mathfrak{a}]}$ as above. To obtain lattices of rank $n = 2m$ we construct 
$$ L^{\text{even}}_{n,[\mathfrak{a}]} := L^{\text{even}}_{2,[\mathfrak{o}]} \perp \ldots \perp L^{\text{even}}_{2,[\mathfrak{o}]} \perp L^{\text{even}}_{2,[\mathfrak{a}]}.$$
For $D \equiv_4 3$, we distinguish between the possible values of the norm residue symbol for $N(\mathfrak{a})$. If $(N(\mathfrak{a}),K/\mathbb{Q})_2 = -1$, then $iii)$ implies $n = 4m+2$. We construct 
$$ L^{\text{even}}_{n,[\mathfrak{a}]} := (L^{\text{even}}_{2,[\mathfrak{a}]} \perp L^{\text{even}}_{2,[\mathfrak{a}^{-1}]}) \perp \ldots \perp (L^{\text{even}}_{2,[\mathfrak{a}]} \perp L^{\text{even}}_{2,[\mathfrak{a}^{-1}]}) \perp  L^{\text{even}}_{2,[\mathfrak{a}]}.$$

If $(N(\mathfrak{a}),K/\mathbb{Q})_2 = 1$, then $iii)$ implies $n = 4m$. If $D \neq -1$, Corollary \ref{cor:ideal:-1at2} implies that we can choose another ideal class $[\mathfrak{b}]$ with $(N(\mathfrak{b}),K/\mathbb{Q})_2 = -1$ and construct 
$$ L^{\text{even}}_{n,[\mathfrak{a}]} := (L^{\text{even}}_{2,[\mathfrak{b}]} \perp L^{\text{even}}_{2,[\mathfrak{b}^{-1}]}) \perp \ldots \perp (L^{\text{even}}_{2,[\mathfrak{b}]} \perp L^{\text{even}}_{2,[\mathfrak{b}^{-1}]}) \perp  L^{\text{even}}_{2,[\mathfrak{b}]} \perp L^{\text{even}}_{2,[\mathfrak{b}^{-1}\mathfrak{a}]}.$$

  Finally if $D = -1$, let $L^{\text{even}}_{4,[\mathfrak{o}]} := \mathfrak{o}x_1+\mathfrak{o}x_2+\mathfrak{o}x_3+\mathfrak{o}x_4$ on the hermitian space given by the following Gram matrix (in the basis $x_1,\ldots,x_4$):
  \begin{align*}
  \begin{pmatrix} 2 & 1 & 0 & 0 \\ 1 & 2 & 1+i & 1+i \\ 0 & 1-i & 2 & 1 \\ 0 & 1-i & 1 & 2 \end{pmatrix}
  \end{align*}
It is easily verified that $L^{\text{even}}_{4,[\mathfrak{o}]}$ is an even  definite unimodular lattice. For arbitrary $n=4m$ we construct
$$ L^{\text{even}}_{n,[\mathfrak{o}]} := L^{\text{even}}_{4,[\mathfrak{o}]} \perp \ldots \perp L^{\text{even}}_{4,[\mathfrak{o}]}. $$
In every case the lattice $L^{\text{even}}_{n,[\mathfrak{a}]}$ is an even  definite unimodular lattice with Steinitz class $[\mathfrak{a}]$ as claimed.
\end{proof}

\begin{corollary}\label{free:quaternary:lattice}
Let $K$ be an imaginary-quadratic field such that $d_K$ is even. Then there is a free quaternary even  definite unimodular lattice over $K$.\qed
\end{corollary}

\section{Distribution of Steinitz classes} \label{Steinitz:Class:Number}
In this section we investigate the question whether $h_{[\mathfrak{a}]} = h_{[\mathfrak{b}]}$ for $[\mathfrak{a}],[\mathfrak{b}] \in \mathcal{C}_K$ holds. Since ideals and (odd) lattices of rank $1$ happen to be essentially the same objects, we have $h_{[\mathfrak{a}]}=1$ for all $\mathfrak{a}\in \mathcal{C}_K$ and therefore the problem is trivial in this case. For higher rank, we can show that in many cases such an equality arises from algebraic constructions. In particular in the binary case such a relation holds for all $[\mathfrak{a}],[\mathfrak{b}]$ in the same ideal genus as Corollary \ref{Steinitz:Class:Number:Equality:binary_2elementary} will show. Using results of \cite{hashimoto23}, we can deduce the same result in the case of ternary lattices. One might expect that this phenomenon is true in general, but the following counter-example for a genus of rank $4$ shows that this is not the case:

\begin{example} \label{counter:example:1} Let $K=\mathbb{Q}(\sqrt{-39})$. The class group of $K$ is the cyclic group of order 4 with generator $[\mathfrak{a}]$.  Because $e(\mathcal{G}) = 1$ (cf. Lemma \ref{genus:partition}) for both genera $\mathcal{G}$ of odd  definite unimodular lattices of rank $4$, the special genera coincide with the partial genera. In the following table, they are indexed by the Steinitz class of the lattices they contain:
\begin{center}
\renewcommand{\arraystretch}{1.2}
\begin{tabular}{|cr|cccc|cccc|}
 \hline

  & & $\mathcal{G}_{[\mathfrak{o}]}$ & $\mathcal{G}_{[\mathfrak{a}]}$ &
$\mathcal{G}_{[\mathfrak{a}^2]}$ & $\mathcal{G}_{[\mathfrak{a}^3]}$ &
$\mathcal{G}_{[\mathfrak{o}]}$ & $\mathcal{G}_{[\mathfrak{a}]}$ &
$\mathcal{G}_{[\mathfrak{a}^2]}$ & $\mathcal{G}_{[\mathfrak{a}^3]}$\\
  & & \multicolumn{4}{|c|}{partial} & \multicolumn{4}{|c|}{special}  \\
\hline
 $\mathbb{I}_{4,\llbracket \mathfrak{o} \rrbracket}$ & class number &
162 & - & 164 & - & 172 & - & 176 & -\\
  & mass & $\tfrac{935}{72}$ & - & $\tfrac{935}{72}$ & - &
$\tfrac{935}{36}$ & - & $\tfrac{935}{36}$ & -  \\
 \hline
 $\mathbb{I}_{4,\llbracket \mathfrak{a} \rrbracket}$ & class number &
- & 152 & - & 152 & - & 152 & - & 152 \\
   & mass & - & $\tfrac{154}{15}$ & - & $\tfrac{154}{15}$ & - &
$\tfrac{308}{15}$ & - & $\tfrac{308}{15}$ \\
 \hline
\end{tabular}
\end{center}
\end{example}

The above example is a minimal one in the sense that the discriminant $d_K=-39$ is the discriminant of smallest absolute value among all discriminants of imaginary-quadratic fields for which there is a genus of  definite unimodular lattices containing partial genera with non-equal partial class numbers. There are also examples of even unimodular lattices of rank $4$ where such a phenomenon occurs, like e.g. the genus $\mathbb{II}_{4,\llbracket\mathfrak{o}\rrbracket}$.\\
We first of all want to explain the equality of the partial class numbers within the genus $\mathbb{I}_{4,\llbracket \mathfrak{a} \rrbracket}$. We denote by $\overline{L}$ the lattice $\overline{\mathfrak{a}}_1x_1+\ldots+\overline{\mathfrak{a}}_nx_n$ on the hermitian space $(V,\overline{h})$, where $\overline{h}(x,y):=\overline{h(x,y)}$.

\begin{lemma}\label{Steinitz:Class:Number:Conjugate}
The map $L\mapsto \overline{L}$ maps isometry classes of  definite unimodular lattices to isometry classes of  definite unimodular lattices preserving the rank, the parity and the unitary group.\qed
\end{lemma}

\begin{proposition}\label{bijectionbar}
Let $\mathcal{G}$ be a genus of  definite unimodular lattices of rank $n$ over $K$ and $[\mathfrak{a}]\in\mathcal{C}_K$. Then the map $L\mapsto \overline{L}$ induces a bijection between $\mathcal{G}_{[\mathfrak{a}]}$ and $\mathcal{G}_{[\overline{\mathfrak{a}}]}$, i.e. $h_{[\mathfrak{a}]}=h_{[\overline{\mathfrak{a}}]}$. \qed
\end{proposition}

Now, let $L=\mathfrak{a}_1x_1+\ldots+\mathfrak{a}_nx_n$ be a unimodular lattice on the hermitian space $(V,h)$. Denote by $^\mathfrak{a} L$ the lattice $\mathfrak{a}L$ on the hermitian space $(V,\tfrac{1}{N(\mathfrak{a})}h)$. The following lemma is well-known and straightforward to check (cf. \cite{hoffmann}):
\begin{lemma}\label{Lemma:bijection:Scaling} 
The map $L\mapsto {^\mathfrak{c}L}$ maps isometry classes of  definite unimodular lattices to isometry classes of  definite unimodular lattices preserving the rank, the parity and the unitary group.\qed
\end{lemma}

\begin{proposition}\label{bijection}
Let $\mathcal{G}$ and $\mathcal{G}'$ be two genera of  definite unimodular lattices over $K$ of rank $n$ and same parity. If there is a $[\mathfrak{c}]\in\mathcal{C}_K$ with $[\mathfrak{b}]=[\mathfrak{c}]^n\cdot [\mathfrak{a}]$, then the map $L\mapsto {^\mathfrak{c}L}$ induces a bijection between $\mathcal{G}_{[\mathfrak{a}]}$ and $\mathcal{G}'_{[\mathfrak{b}]}$, i.e. $h_{[\mathfrak{a}]}=h_{[\mathfrak{b}]}$.  
\end{proposition}

\begin{proof}
 One immediately sees that $\mathfrak{st}(^\mathfrak{c} L) = [\mathfrak{c}]^n \mathfrak{st}(L)$ where $n$ is the rank of $L$. The rest is covered by the above Lemma.
\end{proof}

\begin{theorem}\label{Steinitz:Class:Number:Equality}
 For  definite unimodular lattices of fixed rank $n$ and fixed parity we have:
 \begin{enumerate}
  \item[a)] If $\gcd(n,h_K) = 1$, then $h_{[\mathfrak{a}]}=h_{[\mathfrak{b}]}$ for all $[\mathfrak{a}],[\mathfrak{b}] \in \mathcal{C}_K$.
  \item[b)] If $\gcd(n,\operatorname{exp}(\mathcal{C}_K)) \leq 2$, then $h_{[\mathfrak{a}]}=h_{[\mathfrak{b}]}$ for all $[\mathfrak{b}]\in \llbracket\mathfrak{a}\rrbracket$.
 \end{enumerate}
\end{theorem}
\begin{proof}
 $a)$ Using Lemma \ref{bijection}, we have to show that for all $[\mathfrak{a}],[\mathfrak{b}] \in \mathcal{C}_K$ there exists $[\mathfrak{c}] \in \mathcal{C}_K$, such that $[\mathfrak{b}]=[\mathfrak{c}]^n\cdot [\mathfrak{a}]$, which is equivalent to $\mathcal{C}^n_K = \mathcal{C}_K$. Now the assumption $\gcd(n,h_K) = 1$ implies this equality, as taking $n$-th powers becomes an automorphism of the finite abelian group $\mathcal{C}_K$. 

 $b)$ We have to show that for all $[\mathfrak{b}]\in \llbracket\mathfrak{a}\rrbracket$ there is a $[\mathfrak{c}]\in\mathcal{C}_K$ with $[\mathfrak{b}]=[\mathfrak{c}]^n\cdot [\mathfrak{a}]$. Since $[\mathfrak{a}]$ and $[\mathfrak{b}]$ are in the same ideal genus, this is equivalent to $\mathcal{C}_K^2 \subseteq \mathcal{C}_K^n$, which is essentially a group theoretic question. \\ 
 We use a decomposition of $\mathcal{C}_K$ into cyclic $p$-groups. Since $\exp(\mathcal{C}_K)$ is precisely the product of the largest orders of those cyclic factors, it follows from the assumption $\gcd(n,\operatorname{exp}(\mathcal{C}_K)) \leq 2$ that $\gcd(n,|C|) \le 2$ for each such $C$.
 For $p$ odd or $p = 2$ and $n$ odd we have $C^n = C$ by the same argument as in a) and therefore $C^2 \subseteq C = C^n$ in these cases. Now assume that $p = 2$ and $n$ is even. Then $|C| \leq 2$ or $\frac{n}{2}$ is odd. But then $C^2 = \{1\} \subseteq C^n$ is trivially in the first case and in the second case we have  $\gcd(\frac{n}{2},|C|) = 1$, which then implies $C^{\frac{n}{2}} = C$, thus $C^{n} = C^2$. \\
 Now since $C^2 \subseteq C^n$ for every cyclic factor, we conclude $\mathcal{C}_K^2 \subseteq \mathcal{C}_K^n$ as proposed.
 \end{proof}

We now consider the two trivial cases in which the condition of part $b)$ of the above Theorem is satisfied.
\begin{corollary}\label{Steinitz:Class:Number:Equality:binary_2elementary}
The isometry classes of  definite unimodular lattices of rank $n$ in the same genus, with given Steinitz classes $[\mathfrak{a}],[\mathfrak{b}]$, are in bijection if 
\begin{enumerate}
  \item[i)] $\mathcal{C}_K$ is $2$-elementary, or
  \item[ii)] $n = 2$.
\end{enumerate} 
\end{corollary}

In Proposition $2.8$ of \cite{hashimoto} it is stated that $|\mathcal{G}_{\mathfrak{b}} / U(V)^u|$ is independent of the choice of $\mathfrak{b}$ if $\gcd(n,h_K/2^{t-1}) = 1$, which our result therefore refines, considering that $h_{[a]} = |\mathcal{G}_{[a]}/U(V)| = |\mathcal{G}_{\mathfrak{b}} / U(V)^u|$ for suitable $\mathfrak{b}$ as in Lemma \ref{partial:special}.\\
Examples not covered by Theorem \ref{Steinitz:Class:Number:Equality} respectively Corollary \ref{Steinitz:Class:Number:Equality:binary_2elementary} are the following:
\begin{example}\label{Example:rank3}
a) Let $K=\mathbb{Q}(\sqrt{-23})$. The class group of $K$
is the cyclic group of order 3 with generator $[\mathfrak{a}]$. The genus
$\mathbb{I}_{3,\llbracket\mathfrak{o}\rrbracket}$ is partitioned into 3 times 10
isometry classes with respectively equal Steinitz classes. However, if we count the number of classes in the three partial genera with same decomposition type, we obtain the following table:
\begin{center}
\begin{tabular}{|l|ccc|}
\hline
 & 1+1+1 & 2+1 & 3 \\
\hline
$\mathcal{G}_{[\mathfrak{o}]}$ & 4 & 3 & 3\\
$\mathcal{G}_{[\mathfrak{a}]}$ & 3 & 3 & 4\\
$\mathcal{G}_{[\mathfrak{a}]^2}$ & 3 & 3 & 4\\
\hline
\end{tabular}
\end{center}
b) Let $K=\mathbb{Q}(\sqrt{-83})$, so
$\mathfrak{o}=\mathbb{Z}[\tfrac{1+\sqrt{-83}}{2}]$. The class group of
$K$ is the cyclic group of order 3 with generator $[\mathfrak{a}]$. The
genus $\mathbb{I}_{3,\llbracket\mathfrak{o}\rrbracket}$ is partitioned into 3 times
43 isometry classes with respectively equal Steinitz classes, however
the sequence of orders of the associated unitary groups do only match
for two of these Steinitz classes:
\begin{center}
\begin{tabular}{|l|lllllll|}
\hline
 & 2 & 4 & 8 & 12 & 16 & 24 & 48 \\
\hline
$\mathcal{G}_{[\mathfrak{o}]}$ & 6 & 21 & 10 & 0 & 0 & 3 & 3\\
$\mathcal{G}_{[\mathfrak{a}]}$ & 7 & 18 & 9  & 3 & 3 & 3 & 0\\
$\mathcal{G}_{[\mathfrak{a}]^2}$ & 7 & 18 & 9  & 3 & 3 & 3 & 0\\
\hline
\end{tabular}
\end{center}
\end{example}

In contrast to the bijections $L\mapsto {^{\mathfrak{a}}L}$ resp. $L\mapsto \overline{L}$ given above, these examples illustrate that there are cases in which there cannot exist any bijection preserving the algebraic structure, i.e. for example the decomposition type or the unitary groups. However, for ternary lattices we can explain the existence of some bijection.
\begin{proposition} \label{odd:Steinitz}
 For  definite unimodular lattices of odd rank we have that either $h_{[\mathfrak{a}]} = h_{[\mathfrak{b}]}$ for all $[\mathfrak{a}],[\mathfrak{b}] \in \mathcal{C}_K$ or there is a violation to this inside of each ideal genus.
\end{proposition}

\begin{proof}
 Let $\llbracket\mathfrak{a}\rrbracket,\llbracket\mathfrak{b}\rrbracket$ be distinct ideal genera. Then there exists $[\mathfrak{c}]$ such that  $[\mathfrak{a}] = [\mathfrak{c}][\mathfrak{b}]$. In that case $\llbracket\mathfrak{c}^n\mathfrak{a}\rrbracket = \llbracket\mathfrak{c}\mathfrak{a}\rrbracket = \llbracket\mathfrak{b}\rrbracket$, since $n$ is odd. Therefore multiplication by $[\mathfrak{c}]^n$ induces a bijection between $\llbracket\mathfrak{a}\rrbracket$ and $\llbracket\mathfrak{b}\rrbracket$. By Lemma \ref{bijection} we therefore have a bijection of unimodular lattices with Steinitz class $[\mathfrak{a}'] \in \llbracket\mathfrak{a}\rrbracket$, and Steinitz class $[\mathfrak{b}'] \in \llbracket\mathfrak{b}\rrbracket$, so that $h_{[\mathfrak{a}']} = h_{[\mathfrak{b}']}$. So if $h_{[\mathfrak{a}']} = h_{[\mathfrak{a}]}$ holds for all $[\mathfrak{a}'] \in \llbracket\mathfrak{a}\rrbracket$, it holds for all classes in $\mathcal{C}_K$. On the converse if it is not satisfied for one ideal genus, it cannot 
be for any other.
\end{proof}

\begin{corollary}\label{Steinitz:Class:Number:Equality:ternary}
 For  definite unimodular lattices of rank $3$ over $K$ we have $h_{[\mathfrak{a}]} = h_{[\mathfrak{b}]}$ for all $[\mathfrak{a}],[\mathfrak{b}] \in \mathcal{C}_K$.
\end{corollary}
\begin{proof}
From \cite{hashimoto23} Theorem $5.2$ we know that the special class numbers of all special classes contained in the genus $\mathbb{I}_{3,\llbracket\mathfrak{o}\rrbracket}$ are equal. By Corollary \ref{class:number:coincide} we know that special and partial class numbers coincide in this case. Thus we have $h_{[\mathfrak{a}]} = h_{[\mathfrak{b}]}$ for all $[\mathfrak{a}],[\mathfrak{b}] \in \llbracket \mathfrak{o}\rrbracket$. Using Proposition \ref{odd:Steinitz} we get the claimed result.
\end{proof}

We have seen that for lattices of rank at most $3$ we always have an equality of partial class numbers in a fixed genus (cf. Corollaries \ref{Steinitz:Class:Number:Equality:binary_2elementary} and \ref{Steinitz:Class:Number:Equality:ternary}). The underlying bijection quite often comes from some kind of algebraic construction (cf. Lemmata \ref{Steinitz:Class:Number:Conjugate} and \ref{Lemma:bijection:Scaling}), but for rank $3$ not every such bijection is of that kind. In fact there are examples where there cannot exist any bijection preserving the order of the unitary groups or the decomposition types (cf. Example \ref{Example:rank3}). We have also seen that for rank $4$ there is a counter example to the existence of any bijection at all (cf. Example \ref{counter:example:1}).\\ 
Due to the increasing complexity of the necessary computations to enumerate genera of rank $\geq 5$, we have not found any more counter examples by the computational approach. Indeed, in every considered case the partial class numbers coincide. For example the genus $\mathbb{I}_{6,\llbracket \mathfrak{o} \rrbracket}$ over $\mathbb{Q}(\sqrt{-23})$ is a case not covered by the above results, where the partial class numbers of the three partial genera is $2965$, but the sequences of orders of the unitary groups of the partial genera match only for two of the occurring Steinitz classes. Whether this phenomenon occurs only by chance or only for lattices of small rank, or if there is a more profound reason to be unveiled, remains an open question to us.

\section{A mass formula for partial genera} \label{partial:masses}

\begin{theorem} \label{partial:mass} Let $\mathcal{G}$ be a genus of  definite unimodular lattices over $K$ with $\mathcal{G}_{[\mathfrak{a}]}\subseteq\mathcal{G}$. Then the following holds:
\begin{align*}\operatorname{mass}(\mathcal{G})=\frac{h_K}{2^{t-1}}\cdot \operatorname{mass}(\mathcal{G}_{[\mathfrak{a}]}) \end{align*}
In particular, the partial masses are equal for all classes in the ideal genus $\llbracket \mathfrak{a}\rrbracket$.
\end{theorem}
\begin{proof}
Let $L_1,\ldots, L_{h_{[\mathfrak{a}]}}$ be a set of representatives for the classes in $\mathcal{G}_{[\mathfrak{a}]}$ as specified by Lemma \ref{partial:special}. By Lemma \ref{class:partition} the $U(V)^u$ class of any $L_i$ is the disjoint union of exactly $e(L_i)=[\mathfrak{o}^* : \det{U(L_i)}]$ special classes.\\
By Lemma \ref{genus:partition} $\mathcal{G}_{\mathfrak{b}}$ is the disjoint union of exactly $e(\mathcal{G})$ special genera, therefore the number of special classes in $\mathcal{G}_{\mathfrak{b}}$ is finite. Choose any $\mathfrak{b} \in [\mathfrak{o}] \cap I_K^{(1)}$ and let $\Lambda_1\ldots,\Lambda_m$ ($m=\sum_i e(L_i)$) be a set of representatives of the special classes contained in $\mathcal{G}_{\mathfrak{b}}$. Since $\det : U(L_i) \rightarrow \mathfrak{o}^*$ is a group homomorphism with kernel $SU(L_i)$, we get $|\det{U(L_i)}|=[U(L_i):SU(L_i)]$. This gives
$$ \tfrac{e(L_i)}{|SU(L_i)|}=\tfrac{|\mathfrak{o}^*|}{|\det{U(L_i)}|\cdot |SU(L_i)|} = \tfrac{|\mathfrak{o}^*|}{|U(L_i)|}.  $$
Furthermore, since for isometric lattices $L_i\cong \Lambda_j$ the groups $SU(L_i)$ and $SU(\Lambda_j)$ are conjugate (in $U(V)$), the groups $SU(\Lambda_j)$ have the same order $|SU(L_i)|$ for all $j$ with $L_i\cong \Lambda_j$ . Thus we get
\begin{align}\label{relation_mass}
 \begin{split}
 |\mathfrak{o}^*|\cdot \operatorname{mass}(\mathcal{G}_{[\mathfrak{a}]}) &= |\mathfrak{o}^*|\cdot \sum^{h_{[\mathfrak{a}]}}_{i=1} \tfrac{1}{|U(L_i)|} = \sum^{h_{[\mathfrak{a}]}}_{i=1} \tfrac{e(L_i)}{|SU(L_i)|} \\
                                                             &= \sum^{h_{[\mathfrak{a}]}}_{i=1} \sum_{\stackrel{j=1}{\Lambda_j\cong L_i}}^{m} \tfrac{1}{|SU(\Lambda_j)|} =\sum_{j=1}^{m} \tfrac{1}{|SU(\Lambda_j)|}
\end{split}
\end{align}
Since $\Lambda_1,\ldots,\Lambda_m$ is a set of representatives of the special classes contained in $\mathcal{G}_{[\mathfrak{a}]}$, the last sum in (\ref{relation_mass}) is the sum of the special masses of all special genera contained in $\mathcal{G}_{[\mathfrak{a}]}$. By \cite{rehmann} (4.1) we get
that the special mass does not depend on the special genus $\mathcal{G}^{+}$ but only on the genus $\mathcal{G}$. Since there are exactly $e(\mathcal{G})$ 
special genera contained in $\mathcal{G}_{[\mathfrak{a}]}$ all of them having the same special mass, we get the equation
\begin{align*}\operatorname{mass}(\mathcal{G}_{[\mathfrak{a}]})=\tfrac{e(\mathcal{G})}{|\mathfrak{o}^*|} \cdot \operatorname{mass}^+(\mathcal{G}^+). \end{align*}
Therefore also the partial mass does not depend on the Steinitz class $[\mathfrak{a}]$ and since there are exactly $\tfrac{h_K}{2^{t-1}}$ ideal classes in the ideal genus $\llbracket \mathfrak{a} \rrbracket$ we get the claimed formula.
\end{proof}

Using results of Gan and Yu (cf. \cite{gan}) and Cho (cf. \cite{cho}) we now want to give an exact mass formula for the genera $\mathbb{I}_{n,\llbracket \mathfrak{a} \rrbracket}$ and $\mathbb{II}_{n,\llbracket \mathfrak{a} \rrbracket}$ described in chapter 2. The general mass formula for a genus $\mathcal{G}$ of  definite unimodular lattices of rank $n$ over $K$ reads as follows (cf. \cite{gan} Theorem 10.20):
\begin{align}\label{Mass:Formula1111} \operatorname{mass}(\mathcal{G})= 2 \cdot |d_K|^{\frac{n(n+1)}{4}} \cdot \prod_{j=1}^{n} \frac{(j-1)!}{(2\pi)^j} \cdot \prod_{p\in\mathbb{P}} \alpha_p(\mathcal{G})^{-1}, \end{align}
where $\alpha_p(\mathcal{G})$ denotes the local density of the genus $\mathcal{G}$ at $p$. The local density of a unimodular lattice $L$ for $p\in\mathbb{P}$ is given by Theorem 7.3. in \cite{gan} (we use the following convention: $\chi_K^i=1\Leftrightarrow i \text{ even}$ and $\chi_K^i=\chi_K\Leftrightarrow i \text{ odd}$):
\begin{align*}\alpha_p(\mathcal{G}) =  \prod\limits_{i=1}^n (1-\chi^i_K(p)\cdot p^{-i})\cdot \begin{cases} 1 & p\text{ unramified},\\ 2 \cdot \lambda_p(\mathcal{G})^{-1} & p\text{ ramified},\end{cases}\end{align*}
where for ramified primes $p\in\mathbb{P}$
\begin{align} \label{lambda} \lambda_p(\mathcal{G}):= \begin{cases}
 1  & n\text{ odd}, \\
  1+ (N(\mathfrak{a}),K/\Q)_2 \cdot ((-1)^{\frac{p-1}{2}} p)^{-\frac{n}{2}}
 &  n \text{ even, } p\neq 2,  \\
  1-2^{-n}  & n\text{ even, } p=2, \mathcal{G}\text{ odd}, \\
  2^{1-n}  & n\text{ even, } p=2, \mathcal{G}\text{ even, }d_K\equiv_8 4, \\
  2^{-n} (1+ (N(\mathfrak{a}),K/\Q)_2 \cdot ((-1)^{\frac{D-2}{4}} \cdot 2)^{-\frac{n}{2}}) & n\text{ even, } p=2, \mathcal{G}\text{ even, }d_K\equiv_8 0. \\
\end{cases}
\end{align}
In terms of Dirichlet $L$-series attached to $\chi^{j}_K$ the product of the inverse local densities $\alpha_p(\mathcal{G})^{-1}$ over all $p\in \mathbb{P}$ can be written as:
\begin{align}\label{inverse:local:densities} \prod_{p\in\mathbb{P}} \alpha_p(\mathcal{G})^{-1}=2^{-t} \cdot \prod_{j=1}^n L(j,{\chi^j_K}) \cdot \prod_{p|d_K}\lambda_p(\mathcal{G}) \end{align}

For further use, we state:
\begin{proposition}\label{Partial:mass:formula}
Let $\mathcal{G}$ be a genus of  definite unimodular lattices of rank $n$ over $K$ with $\mathcal{G}_{[\mathfrak{a}]}\subseteq\mathcal{G}$. Then
\begin{align*} \operatorname{mass}(\mathcal{G}_{[\mathfrak{a}]})= \frac{1}{|\mathfrak{o}^*|} \cdot |d_K|^{\frac{n^2+n-2}{4}} \cdot \prod_{j=2}^{n} \frac{(j-1)!}{(2\pi)^j} \cdot \prod_{j=2}^n L(j,{\chi^j_K}) \cdot \prod_{p|d_K}\lambda_p(\mathcal{G})
\end{align*}
where $\lambda_p(\mathcal{G})$ is defined in (\ref{lambda}).
 \end{proposition}
\begin{proof} Combining formulas (\ref{Mass:Formula1111}) and (\ref{inverse:local:densities}) yields the following mass formula for the genus $\mathcal{G}$:
$$ \operatorname{mass}(\mathcal{G})= 2^{1-t} \cdot |d_K|^{\frac{n(n+1)}{4}} \cdot \prod_{j=1}^{n} \frac{(j-1)!}{(2\pi)^j} \cdot \prod_{j=1}^n L(j,{\chi^j_K}) \cdot \prod_{p|d_K}\lambda_p(\mathcal{G}) $$
Using Theorem \ref{partial:mass} and the class-number formula $h_K=\tfrac{|\mathfrak{o}^*|}{2\pi} \sqrt{|d_K|} \cdot L(1,\chi_K)$ (cf. \cite{zag} §8 Satz 5) we get exactly the claimed result.
\end{proof}

\begin{corollary}\label{Partial:mass:formula:exact}
Let $\mathcal{G}$ be a genus of  definite unimodular lattices of rank $n$ over $K$ with $\mathcal{G}_{[\mathfrak{a}]}\subseteq\mathcal{G}$. Then
\begin{align*} \operatorname{mass}(\mathcal{G}_{[\mathfrak{a}]})= \frac{1}{|\mathfrak{o}^*|} \cdot \prod_{j=2}^{n} \frac{\left|B_{j,\chi_K^j}\right|}{2j} \cdot\prod_{p|d_K}\lambda_p(\mathcal{G}) \cdot 
\begin{cases} 1 & n \text{ odd,}\\ |d_K|^{n/2} & n \text{ even,}                                                                                                                                                         
\end{cases} 
\end{align*}
\nopagebreak where $B_{j,\chi^j_K}$ denote the generalized Bernoulli numbers and $\lambda_p(\mathcal{G})$ is defined in (\ref{lambda}).
\end{corollary}
\begin{proof}
Using the functional equation for the $L$-series we get 
\begin{align*} |d_K|^{\frac{n(n+1)}{4}} \cdot \prod_{j=1}^{n} \frac{(j-1)!}{(2\pi)^j} \cdot \prod_{j=1}^n L(j,{\chi^j_K}) &= 2^{-n}\cdot \prod_{j=1}^n L(1-j,{\chi^j_K})  \cdot \begin{cases} 1 & n \text{ odd,}\\ |d_K|^{n/2} & n \text{ even.} \end{cases}  
\end{align*}
Using $L(1-j,\chi_K^j)=\frac{\left|B_{j,\chi_K^j}\right|}{j}$ (cf. \cite{washington} ) and Theorem \ref{partial:mass} we finally get the result. 
\end{proof}

\begin{remark}
Proposition \ref{Partial:mass:formula:exact} makes the partial mass computable explicitly since the generalized Bernoulli numbers $B_{j,\chi^j_K}$ are given by the following formula (cf. \cite{washington} Prop. 4.1)
$$ B_{j,\chi^j_K} = \begin{cases} B_j(1) & j\text{ even}, \\ {|d_K|}^{j-1} \sum_{a=1}^{|d_K|} \chi_K(a) B_j(\frac{a}{|d_K|}) & j\text{ odd}, \end{cases} $$
where $B_j(X)$ denotes the usual Bernoulli polynomial.
\end{remark}

\section{Single-class partial genera} \label{class:number:problem}
In this last section we determine all single-class partial genera of  definite unimodular lattices of rank $n\ge 2$ over an imaginary-quadratic field. In particular we determine all single-class genera of such lattices. It turns out that a partial genus consists of a single isometry class if and only if the genus consists of a single isometry class.

\begin{lemma}\label{lambda:bound}
Let $\mathcal{G}$ be a genus of  definite unimodular lattices of rank $n$ over $K$. The following inequality holds:
\begin{align*}
 \prod_{p \mid d_K} \lambda_p(\mathcal{G})  \ge 2^{-n} \cdot \frac{2^{n/2} - 1}{(t+1)^{n/2}}.
\end{align*} 
\end{lemma}
\begin{proof}
It is easy to check by the definition of $\lambda_p(\mathcal{G})$ in (\ref{lambda}) that:
\begin{align}\label{lambda:bound1}
 \prod_{p \mid d_K} \lambda_p(\mathcal{G}) &\ge 2^{-n} \cdot \prod_{p \mid d_K} \left( 1- p^{-\frac{n}{2}} \right)
\end{align} 
If $p_k$ denotes the $k$th prime number, we have $p_k\ge k+1$ and therefore
\begin{align*} 
  \prod_{p \mid d_K} \left( 1- p^{-\frac{n}{2}} \right) \ge \prod_{k=1}^{t} \left( 1- p_k^{-\frac{n}{2}} \right) \ge \prod_{k = 1}^{t} \left( 1- (k+1)^{-\frac{n}{2}} \right) = \prod_{k = 1}^{t}  \frac{(k+1)^{\frac{n}{2}}-1}{(k+1)^{\frac{n}{2}}} 
\end{align*} 
The Bernoulli inequality gives us $\tfrac{(k+1)^{\frac{n}{2}}-1}{k^{\frac{n}{2}}} \ge 1$ for $k\in\{2,\ldots,t\}$ and therefore
$$ \prod_{k = 1}^{t}  \frac{(k+1)^{\frac{n}{2}}-1}{(k+1)^{\frac{n}{2}}} = (2^{\frac{n}{2}}-1) \cdot \frac{3^{\frac{n}{2}}-1}{2^{\frac{n}{2}}} \cdots \frac{(t+1)^{\frac{n}{2}}-1}{t^{\frac{n}{2}}} \cdot \frac{1}{(t+1)^{\frac{n}{2}}} \ge \frac{2^{n/2} - 1}{(t+1)^{n/2}}$$
Together with (\ref{lambda:bound1}) we get the claimed result.
\end{proof}

\begin{proposition}
For all $n\ge2 $ there is a $d_{\text{max}}(n)\in \mathbb{N}$ such that for all single-class partial genera of  definite unimodular lattices of rank $n$ over $K$ we have $|d_K|\leq d_\text{max}(n)$. For $2\le n \le 7$ such a bound $d_\text{max}(n)$ is given by:
\begin{center}
\begin{tabular}{|c|c|c|c|c|c|c|}
\hline
 $n$ & 2 & 3 & 4 & 5 & 6 & 7 \\
 \hline
 $d_\text{max}(n)$ & 312 & 16 & 8 & 7 & 4  & 3  \\
 \hline
\end{tabular}
\end{center}
\end{proposition}
\begin{proof}
For a single-class partial genus $\mathcal{G}_{[\mathfrak{a}]}$ with representative $L$ we have $\operatorname{mass}(\mathcal{G}_{[\mathfrak{a}]})=\tfrac{1}{|U(L)|}$. Since $|U(L)| \geq |\mathfrak{o}^*|$ for every lattice $L$, we necessarily have $\operatorname{mass}(\mathcal{G}_{[\mathfrak{a}]})\le \tfrac{1}{|\mathfrak{o}^*|}$ for a single-class partial genus $\mathcal{G}_{[\mathfrak{a}]}$. \\ 
We now want to find a lower bound on the partial mass which for fixed rank $n\ge 2$ depends only on the discriminant $d_K$. Therefore we use the expression for the partial mass in Proposition \ref{Partial:mass:formula}. Since $\zeta_K(j)=\zeta(j)\cdot L(j,\chi_K)$ (cf. \cite{zag} p.100) and $\zeta_K(j)\geq 1$ for all $j>1$ we easily deduce $L(j,\chi_K) = \frac{\zeta_K(j)}{\zeta(j)} \ge \frac{1}{\zeta(j)}$.
Using $t \le \log_2(|d_K|)$ and Lemma \ref{lambda:bound} for even $n$, it is easily seen that if 
\begin{align*} f_n(d): = \begin{cases}  \frac{1}{|\mathfrak{o}^*|} \cdot d^{\frac{n^2+n-2}{4}} \cdot \prod_{j=2}^{n} \frac{(j-1)!}{(2\pi)^j}  \cdot \frac{\zeta(2) \cdot \zeta(4) \cdot \ldots \cdot \zeta(n-1)}{\zeta(3) \cdot \zeta(5) \cdot \ldots \cdot \zeta(n)} & n\text{ odd}, \\ \tfrac{1}{|\mathfrak{o}^*|}\cdot d^{\frac{n^2+n-2}{4}} \cdot \frac{2^{n/2} - 1}{2^n(\log_2(|d_K|)+1)^{\frac{n}{2}}} \cdot  \prod_{j=2}^{n} \frac{(j-1)!}{(2\pi)^j}  \cdot \tfrac{\zeta(2) \cdot \zeta(4) \cdot \ldots \cdot \zeta(n)}{\zeta(3) \cdot \zeta(5) \cdot \ldots \cdot \zeta(n-1)} & n\text{ even}, \end{cases}
\end{align*} 
we have $f_n(|d_K|)\leq \operatorname{mass}(\mathcal{G}_{[\mathfrak{a}]})$  for all $n\in \N$. Since $f_n : \N \rightarrow \R_{\geq 0}$ is a strictly increasing function (in $d$) for $n\geq 2$, there is a $d_{\text{max}}(n)$ such that $f_n(d)\le \tfrac{1}{|\mathfrak{o}^*|}$ implies $d\le d_{\text{max}}(n)$. In particular we have $|d_K|\leq d_{\text{max}}(n)$ if there is a single-class partial genus of  definite unimodular lattices of rank $n$ over $K$.\\
For fixed $n$ we can estimate this bound $d_{\text{max}}(n)$. Using Corollary \ref{Partial:mass:formula:exact} we can compute the actual partial masses explicitly. Therefore, once given a bound $d_\text{max}(n)$, we can try to improve it by substituting it by the greatest number $d\leq d_{\text{max}}(n)$ such that $d$ is a discriminant of some imaginary-quadratic field $K$ and there is a partial genus of  definite unimodular lattices of rank $n$ over $K$ with partial mass of the form $\tfrac{1}{2m}$ for some $m\in\mathbb{N}$.
\end{proof}
Instead of computing $d_\text{max}(n)$ for all $n\geq 2$, we will effectively only need bounds for $n\in\{2,4\}$. All other cases will be solved by constructing lattices with higher rank from other ones with lower rank. The procedure is as follows:\\

As seen in section 2, there is a free odd definite unimodular lattice $L^{\text{odd}}_{1,[\mathfrak{o}]}$ of rank $1$ over every imaginary-quadratic field. Now the cancellation law implies that the mapping $L\mapsto L\perp L^{\text{odd}}_{1,[\mathfrak{o}]}$ is an injection preserving the Steinitz class. Thus, if the partial genus of odd definite unimodular lattices of rank $n$ consists of more than one single class, then the corresponding partial genus of odd definite unimodular lattices with the same Steinitz class and rank $n+1$ also consists of more than one class. In order to find all single-class partial genera we therefore have to determine the single-class partial genera of rank $2$. Only for those we have to consider the corresponding partial genera of higher rank until we reach a partial genus consisting of more than one class.\\

We use essentially the same argument to determine all partial genera of even lattices. The procedure is slightly more difficult since even definite unimodular lattices exist only for even rank and there is a free even definite unimodular lattice $L^{\text{even}}_{4,[\mathfrak{o}]}$ of rank $4$ for every imaginary-quadratic field (cf. Corollary \ref{free:quaternary:lattice}), but not necessarily of rank $2$. Therefore we have to determine all single-class partial genera of even definite unimodular lattices rank $2$ and $4$ to proceed as above.\\

To check whether a partial genus consists of only one class, we compute the actual partial mass using Corollary \ref{Partial:mass:formula:exact}. If this mass is of the form $\frac{1}{2m}$ for some $m \in \N$, the corresponding partial genus has to be tested further. For such a remaining partial genus $\mathcal{G}_{[\mathfrak{a}]}$ we construct a lattice $L$ in it using the proof of Theorem \ref{start:lattices}. Then $\mathcal{G}_{[\mathfrak{a}]}$ is a single-class partial genus if and only if $\operatorname{mass}(\mathcal{G}_{[\mathfrak{a}]})=\frac{1}{|U(L)|}$. Doing this we get a complete list of all single-class partial genera of definite unimodular lattices of rank $n\geq 2$ over an imaginary-quadratic field. This proves the following Theorem:

\begin{theorem}
 For a genus $\mathcal{G}$ of  definite unimodular lattices of rank $n\geq 2$ over an imaginary-quadratic field the following are equivalent:
 \begin{enumerate}
  \item[a)] A partial genus $\mathcal{G}_{[\mathfrak{a}]}$ contained in $\mathcal{G}$ consists of only one isometry class.
  \item[b)] $\mathcal{G}$ consists of only one isometry class.
  \item[c)] $\mathcal{G}$ is included in the following list:
   \begin{flushleft}
   \begin{tabular}{c|l}
  $d_K$ & genera\\
  \hline
  $-3$ & $\mathbb{I}_{2,\llbracket\mathfrak{o}\rrbracket}$, $\mathbb{I}_{3,\llbracket\mathfrak{o}\rrbracket}$, $\mathbb{I}_{4,\llbracket\mathfrak{o}\rrbracket}$, $\mathbb{I}_{5,\llbracket\mathfrak{o}\rrbracket}$\\
  $-4$ & $\mathbb{I}_{2,\llbracket\mathfrak{o}\rrbracket}$, $\mathbb{I}_{3,\llbracket\mathfrak{o}\rrbracket}$, $\mathbb{I}_{4,\llbracket\mathfrak{o}\rrbracket}$, $\mathbb{II}_{4,\llbracket\mathfrak{o}\rrbracket}$\\
  $-7$ & $\mathbb{I}_{2,\llbracket\mathfrak{o}\rrbracket}$\\
  $-8$ & $\mathbb{I}_{2,\llbracket\mathfrak{o}\rrbracket}$, $\mathbb{II}_{2,\llbracket\mathfrak{o}\rrbracket}$\\
  $-20$ & $\mathbb{I}_{2,\llbracket\mathfrak{a}\rrbracket}$\\
  $-24$ & $\mathbb{II}_{2,\llbracket\mathfrak{o}\rrbracket}$, $\mathbb{II}_{2,\llbracket\mathfrak{a}\rrbracket}$\\
  $-40$ & $\mathbb{II}_{2,\llbracket\mathfrak{o}\rrbracket}$, $\mathbb{II}_{2,\llbracket\mathfrak{a}\rrbracket}$\\
  $-88$ & $\mathbb{II}_{2,\llbracket\mathfrak{a}\rrbracket}$\\

 \end{tabular}
 \end{flushleft}
 In each case $[\mathfrak{a}]$ is the unique non-trivial ideal class, and $\llbracket\mathfrak{a}\rrbracket$ denotes the unique non-trivial ideal genus.
 \end{enumerate}\qed
\end{theorem}

\begin{remark} As already mentioned, ideals and lattices of rank $1$ happen to be essentially the same objects. Therefore the problem of finding all single class partial genera of  definite unimodular lattices of rank $n=1$ is trivial. \\
However the more familiar problem of finding all genera consisting of a single isometry class of lattices is in fact the problem on \emph{idoneal numbers}: Finding all fundamental discriminants, such that each ideal genus consists of a single ideal class, i.e. $h_K=2^{t-1}$. This is a long standing open problem and a solution lies well beyond the scope of the techniques used here.
\end{remark}

\bibliographystyle{amsalpha}
\bibliography{refs}

\end{document}